\documentclass[11pt,reqno]{amsart}
\usepackage{amssymb,latexsym}
\usepackage{graphicx}
\usepackage{fancyhdr}
\usepackage{multirow}

\newcommand{\N}{\mathbb N}
\newcommand{\C}{\mathbb C}
\newcommand{\up}{\textup}

\newcommand{\R}{\mathbb R}
\newcommand{\onetok}{\{1,2,\ldots,k\}}

\theoremstyle{plain}
\numberwithin{equation}{section}
\newtheorem{thm}{Theorem}[section]
\newtheorem{theorem}[thm]{Theorem}

\newtheorem{lem}[thm]{Lemma}
\newtheorem{defn}[thm]{Definition}

\title{Integer compositions with part sizes not exceeding $k$}
\author{Martin E. Malandro}
\address{Box 2206\\
Department of Mathematics and Statistics\\
Sam Houston State University\\
Huntsville, TX\\
77341-2206, USA}
\email{malandro@shsu.edu}
\date{\today}
\keywords{Restricted composition, asymptotic, rhythm pattern, analytic combinatorics}
\subjclass[2010]{05A16, 05A15, 11B37}

\begin{document}

%
%
%
%
%
%
%
%
%
%
%
%
%
%
%

\setcounter{page}{1}

\begin{abstract}
We study the compositions of an integer $n$ whose part sizes do not exceed a fixed integer $k$. We use the methods of analytic combinatorics to obtain precise asymptotic formulas for the number of such compositions, the total number of parts among all such compositions, the expected number of parts in such a composition, the total number of times a particular part size appears among all such compositions, and the expected multiplicity of a given part size in such a composition. Along the way we also obtain recurrences and generating functions for calculating several of these quantities. 
Our results also apply to questions about certain kinds of tilings and rhythm patterns.
\end{abstract}

\maketitle

\section{Introduction}

Let $k$ be a fixed positive integer. We study numerical and combinatorial aspects of the sequence $\{F_n\}_{n=-\infty}^\infty$, defined by 
\begin{equation}
F_n=\begin{cases}
0, & \up{if } n<0;\\
1, & \up{if }n=0;\\
\sum_{a=1}^k F_{n-a}, & \up{if } n> 0,
\end{cases}
\label{eqFnRecurIntro}
\end{equation}
and related sequences. When $k=2$, the numbers $F_n$ form the (shifted) Fibonacci sequence. The numbers $F_n$ are known to count many things, including the tilings of the $1\times n$ rectangle using tiles of length no more than $k$ \cite{Utilitas, ProofsThatCount}, the rhythm patterns of length $n$ with note lengths not exceeding $k$ (Section \ref{SecTilings}, \cite{HallRhythm}), the new cells after the $n$th step in a budding process that starts with one cell and where each cell has one child per step, up to a total of $k$ children per cell \cite{Yeast}, and the compositions of an integer $n$ with part sizes no larger than $k$ (Section \ref{SecRecurrences}, \cite{CompsWithPartsInASet}).

A {\em composition} of a positive integer $n$ is a sequence of positive integers $(\lambda_1,\lambda_2,\ldots,\lambda_m)$ such that $\sum_{j=1}^m \lambda_j = n$. The four compositions of $n=3$, for instance, are (1,1,1), (1,2), (2,1), and (3). It is well known and easy to prove that a positive integer $n$ has $2^{n-1}$ compositions \cite{CombinatoricsOfCompsAndWords}. The integers $\lambda_j$ of a composition are its {\em parts}. 
If $L$ is a set of positive integers, we can consider the set of compositions of $n$ such that every part of each composition lies in $L$. Such compositions are called {\em restricted compositions} or {\em $L$-compositions} of $n$. 

The literature on unrestricted integer compositions is vast, and restricted compositions have also received a fair amount of attention. For example, $\{1,2\}$-compositions were studied in \cite{OnesAndTwos}, $\{1,k\}$-compositions were studied in \cite{OneK}, and $L$-compositions in general were studied in \cite{CompsWithPartsInASet}. See also \cite{PatternsLength3, CombinatoricsOfCompsAndWords, Hit3, Hit2, Hit1, PatternAvoidance}. In addition, both restricted and unrestricted integer compositions are examples of sequence constructions in symbolic combinatorics, so several very general theorems from this field apply to them \cite{AC}. 

In this paper we fix $k$, let $L=\onetok$, and study the set of $L$-compositions of $n$, i.e., the set of compositions of $n$ whose part sizes do not exceed $k$. We use the following notation.

\begin{itemize}
	\item $F_n$ denotes the number of $\onetok$-compositions of $n$. (We show that this definition of $F_n$ agrees with (\ref{eqFnRecurIntro}) in Section \ref{SecRecurrences}.)
	\item $T_n$ denotes the total number of parts among all $\onetok$-compositions of $n$.
	\item $A_n=T_n/F_n$ denotes the average number of parts of a $\onetok$-composition of $n$.
	\item For $j\in\onetok$, $C_{n,j}$ denotes the total number of times the part size $j$ appears among all $\onetok$-compositions of $n$.
	\item For $j\in\onetok$, $A_{n,j}=C_{n,j}/F_n$ denotes the average number of times the part size $j$ appears in a $\onetok$-composition of $n$.
\end{itemize}

In this paper we study these five quantities in a unified way, with the ultimate goal of giving precise asymptotic formulas for all of them. When we say that our asymptotic formulas are {\em precise}, we mean that both their percentage errors and their absolute errors go to $0$ as $n\rightarrow \infty$. In fact, the absolute errors for our asymptotic formulas decay at an exponential rate. 

Previously, Flores \cite{Flores} gave a precise asymptotic for $F_n$, which he derived using the theory of recurrence relations whose characteristic polynomials have no repeated roots. Then, a precise asymptotic for $A_n$ was obtained in \cite[Theorem V.1]{AC} using the methods of analytic combinatorics. Combining these immediately results in a precise asymptotic for $T_n$. Our asymptotic for $C_{n,j}$ is new, and our asymptotic for $A_{n,j}$ improves on the one in \cite[Theorem V.2]{AC}. 

Along the way to developing our asymptotic formulas, we also give recursive methods for computing all five of these quantities and generating functions for $F_n$, $T_n$, and $C_{n,j}$, some of which have appeared previously. Since one of our goals is to provide a unified treatment of these five quantities, we provide proofs of all of our results, including the ones that are not new. In addition, our approach provides an alternate route to Flores's asymptotic for $F_n$.

We use the following notation. Let $G(x)=\sum_{a=1}^k x^a$ and let $\phi$ denote the unique positive real solution to
\begin{equation}
\label{eqDefOfPhi}
\frac{1}{\phi^{1}} + \frac{1}{\phi^{2}} + \cdots + \frac{1}{\phi^{k}} = 1.
\end{equation}
(That there is a unique positive real solution is guaranteed by Descartes' rule of signs.) Let $G'$ and $G''$ denote the ordinary first and second derivatives of $G$, and let $\sigma=1/\phi$. When $k\geq 2$ we have $\phi>1$, and when $k=2$, $\phi$ is the golden ratio. Here are the asymptotic formulas we derive in this paper. (We review big-$O$ notation in Section \ref{SecAsymptotics}.) 

\begin{thm}\label{ThmMyResults}For some constant $A$ with $0<A<1$ we have
\begin{enumerate}
	\item $\displaystyle F_n = \frac{\phi^{n+1}}{G'(\sigma)} + O(A^n) \quad$ \cite{Flores},
	\item $\displaystyle T_n = \frac{\phi^{n+2}}{G'(\sigma)^2}(n+1) + \frac{\phi^{n+1}G''(\sigma)}{G'(\sigma)^3} - \frac{\phi^{n+1}}{G'(\sigma)} +O(A^n)$,
	\item $\displaystyle A_n = \frac{\phi}{G'(\sigma)}(n+1) - 1 + \frac{G''(\sigma)}{G'(\sigma)^2} +O(A^n) \quad$ \cite[Theorem V.1]{AC},
	\item $\displaystyle C_{n,j} =\frac{\phi^{n+2-j}}{G'(\sigma)^2}(n+1-j) + \frac{\phi^{n+1-j}G''(\sigma)}{G'(\sigma)^3}
+O(A^n)$, and
	\item $\displaystyle A_{n,j} = \frac{\phi^{1-j}}{G'(\sigma)}(n+1-j) + \frac{\phi^{-j}G''(\sigma)}{G'(\sigma)^2} + O(A^n)$.
\end{enumerate}
\end{thm}


Actually, the formula for $A_n$ in Theorem \ref{ThmMyResults} follows from an asymptotic formula for supercritical sequences, of which $\onetok$-compositions are one example. See \cite{AC} for more on supercritical sequences.
Previously, the best asymptotic result we could find for $A_{n,j}$ is an asymptotic for supercritical sequences \cite[Theorem V.2]{AC} which, when translated into the setting of $\onetok$-compositions, becomes
\[
A_{n,j} = \frac{\phi^{1-j}}{G'(\sigma)}n + O(1).
\]
Our asymptotic formula for $A_{n,j}$ refines this into a precise asymptotic. Also, we remark that the formula for $F_n$ in Theorem \ref{ThmMyResults} was stated differently in \cite{Flores}, but one may use equation (\ref{eqDefOfPhi}) to verify that the formula here is the same as the one in \cite{Flores}.

The rest of the paper is organized as follows. In Section \ref{SecTilings} we review the connection between rhythm patterns and restricted compositions and we derive our recurrences. In Section \ref{SecGenFns} we use our recurrences to obtain our generating functions. We then analyze those generating functions in Section \ref{SecAsymptotics} to obtain our asymptotic results. Finally, Section \ref{SecExampleValues} contains tables of values of $F_n$, $T_n$, $A_n$, $C_{n,j}$, and $A_{n,j}$ and our asymptotic approximations to them for small values of $n$ and $k$.

\section{Rhythm patterns and recurrences}
\label{SecTilings}
\label{SecRecurrences}

In this section we give recursive methods for calculating the quantities $F_n$, $T_n$, $A_n$, $C_{n,j}$, and $A_{n,j}$. Before doing so, however, we mention that our original motivation for studying restricted compositions was musical. The $L$-compositions of an integer $n$ are in bijection with the rhythm patterns of length $n$ with note lengths in $L$ \cite{HallRhythm}. A {\em rhythm pattern} of length $n$ is a sequence of notes (also called hits) and rests played over $n$ evenly spaced pulses, where a note occurs on the first pulse, and on each further pulse either a note or a rest (but not both) occurs. 
The {\em length} of a note is 1 plus the number of rests until the next note or the end of the pattern. The bijection is given by mapping compositions to rhythm patterns part by part, where the part $j$ corresponds to a note of length $j$ (i.e., a hit followed by $j-1$ rests). 
For example, the composition $(2,1,1,1,2,2,1)$ of $n=10$ corresponds to the rhythm pattern
\begin{center}
	hit - rest - hit - hit - hit - hit - rest - hit - rest - hit.
\end{center}
Thus, if we consider $\onetok$-compositions, $F_n$ counts the number of rhythm patterns of length $n$ with no note length exceeding $k$, $T_n$ counts how many notes one would play if one played all such patterns one after another, and $A_n$ says how many notes occur on average in such a pattern. $C_{n,j}$ indicates how many times the note length $j$ occurs among all such patterns, and $A_{n,j}$ indicates how many times the note length $j$ occurs on average in such a pattern. 

We now derive our recurrences. Our recurrence for $F_n$ is well known, and previously appeared in \cite{CompsWithPartsInASet,Yeast}, among others.
\begin{theorem}$F_n$ satisfies\label{ThmFnRecur}
\begin{equation}\label{eqFnRecur}
F_n=\begin{cases}
0, & \up{if } n<0;\\
1, & \up{if }n=0;\\
\sum_{a=1}^k F_{n-a}, & \up{if } n> 0.
\end{cases}
\end{equation}
\end{theorem}
\begin{proof}
Certainly $F_n=0$ if $n$ is negative. $F_0=1$ as there is one composition---the empty composition---of $0$. If $n>0$ then we obtain the $\onetok$-compositions of $n$ by appending 1's to the $\onetok$-compositions of $n-1$, by appending 2's to the $\onetok$-compositions of $n-2$, etc., all the way down to appending $k$'s to the $\onetok$-compositions of $n-k$. Thus we have 
$F_n=F_{n-1}+F_{n-2}+\cdots+F_{n-k}$ for $n>0$.
\end{proof}

Next we obtain a recurrence for $C_{n,j}$ (and hence a recursive method for computing $A_{n,j}$).

\begin{thm}\label{ThmCnjRecur}Let $j\in \onetok$. Then
\begin{equation}
\label{eqCjnRecur}
C_{n,j}=
\begin{cases}
0,&\up{if } n\leq 0;\\
F_{n-j} + \sum_{a=1}^kC_{n-a,j},&\up{if }n>0.
\end{cases}
\end{equation}
\end{thm}
\begin{proof}
If $n<j$ we have $C_{n,j}=0$, with which this recurrence agrees. Next, if $n=j$, since $F_0=1$ and $C_{a,j}=0$ for $a<j$, this recurrence correctly yields $C_{n,j}=1$. 

Finally, suppose $n>j$. By induction we may assume that this recurrence correctly computes $C_{a,j}$ for all $a<n$. Split the set of $\onetok$-compositions of $n$ into two classes: those that end with $j$ and those that do not. Among those that end with $j$, there are $(C_{n-j,j}+F_{n-j})$ $j$'s. (Specifically, there are a total of $F_{n-j}$ $j$'s at the ends and a total of $C_{n-j,j}$ $j$'s among all other places of these compositions.) Among the $\onetok$-compositions of $n$ that end with $a\in \onetok$, for $a\neq j$, there are $(C_{n-a,j})$ $j$'s. Summing these counts across $a\in \onetok$ yields $C_{n,j}= F_{n-j} + \sum_{a=1}^kC_{n-a,j}$, as desired.
\end{proof}

Finally we obtain a recurrence for $T_n$  (and hence a recursive method for computing $A_n$).

\begin{thm}\label{ThmTnRecur}$T_n$ satisfies 
\begin{equation}
\label{eqTnRecur}
T_n= 
\begin{cases}
0, & \up{if } n\leq 0;\\
F_n + \sum_{a=1}^kT_{n-a}, & \up{if } n>0.
\end{cases}
\end{equation}
\end{thm}

\begin{proof}$T_0=0$ because there are no parts in the empty composition. If $n>0$, then we have
\begin{align*}
T_n &= \sum_{j=1}^kC_{n,j} \\
&= \sum_{j=1}^k\left(F_{n-j} + \sum_{a=1}^kC_{n-a,j}\right)\\
&=F_n + \sum_{a=1}^k\sum_{j=1}^k C_{n-a,j} \\
&=F_n + \sum_{a=1}^kT_{n-a},
\end{align*}
using the fact that the recurrences $C_{n,j}= F_{n-j} + \sum_{a=1}^kC_{n-a,j}$ and $F_n=\sum_{j=1}^k F_{n-j}$ are both valid for $n>0$.
\end{proof}


\section{Generating functions}
\label{SecGenFns}

In this section we derive generating functions for $F_n$, $T_n$, and $C_{n,j}$. Recall $G(x)=\sum_{a=1}^kx^a$. The generating function for $F_n$ is well known, and previously appeared in \cite{AC,CompsWithPartsInASet}.

\begin{thm}
\label{ThmFnGen}
Let $F(x)$ be the ordinary generating function for the number of $\onetok$-compositions of $n$. That is, $F(x)=F_0+F_1x+F_2x^2+\cdots$. Then
\begin{equation*}
F(x)=\frac{1}{1-G(x)}.
\end{equation*}
\end{thm}

\begin{proof}We prove this directly. Using $F_0=1$ and the recurrence in (\ref{eqFnRecur}) for $n>0$, we have
\begin{align*}
F(x)&=1+\sum_{n=1}^\infty F_n x^n\\
&= 1+ \sum_{n=1}^\infty \sum_{a=1}^kF_{n-a}x^n\\
& = 1+ \sum_{a=1}^k\sum_{n=1}^\infty F_{n-a}x^n\\
&=1+ \sum_{a=1}^kx^a\sum_{n=1}^\infty F_{n-a}x^{n-a}\\
&=1+ \sum_{a=1}^kx^a F(x).
\end{align*}
Solving for $F(x)$ yields
$
F(x)=1/({1-\sum_{a=1}^kx^a}),
$
as desired.
\end{proof}

Next, for $j\in\onetok$, we obtain the generating function for $C_{n,j}$. We remark that this generating function can also be obtained by differentiating formula (7) on p.\ 293 of \cite{AC} with respect to $u$ and evaluating at $u=1$. There the appropriate generating function was built directly from the combinatorics of sequences. Here we derive it from our recursive formula (\ref{eqCjnRecur}) instead.

\begin{thm}
\label{ThmCjnGen}
Let $j\in\onetok$ and let $C_j(x)$ be the ordinary generating function for the total number of occurrences of the part size $j$ among all $\onetok$-compositions of $n$. That is, $C_j(x) = C_{0,j}+C_{1,j}x+C_{2,j}x^2+\cdots$. Then
\begin{equation*}
C_j(x)=\frac{x^j}{\left(1-G(x)\right)^2}.
\end{equation*}
\end{thm}

\begin{proof}Using $C_{0,j}=0$ and equation (\ref{eqCjnRecur}) we have
\begin{align*}
C_j(x) &= \sum_{n=1}^\infty C_{n,j} x^n \\
&= \sum_{n=1}^\infty \left( F_{n-j} + \sum_{a=1}^kC_{n-a,j} \right)x^n\\
&= \sum_{n=1}^\infty F_{n-j}x^n + \sum_{n=1}^\infty\sum_{a=1}^kC_{n-a,j}x^n \\
&=x^j\sum_{n=1}^\infty F_{n-j}x^{n-j} + \sum_{a=1}^kx^a \sum_{n=1} C_{n-a,j} x^{n-a}\\
&=x^j F(x) + \sum_{a=1}^kx^a C_j(x).
\end{align*}
Solving for $C_j(x)$ we obtain $C_j(x)=x^jF(x) / (1-\sum_{a=1}^kx^a) = x^j/(1-\sum_{a=1}^kx^a)^2$, as desired.
\end{proof}

Theorem \ref{ThmCjnGen} also reveals the nice fact that for a fixed value of $k$, the sequences $\{C_{n,j}\}_{n=0}^\infty$ for $j\in \onetok$ are all just shifts of the same sequence. See Table \ref{TableCnjExamples}, for example. 

Finally we obtain the generating function for $T_n$. This generating function previously appeared more generally in the context of sequence constructions on p.\ 178 of \cite{AC}. Here we derive it using our recursive formula (\ref{eqTnRecur}) instead. 

\begin{thm}
\label{ThmTnGen}
Let $T(x)$ be the ordinary generating function for the total number of parts among all $\onetok$-compositions of $n$. That is, $T(x)=T_0 + T_1 x + T_2 x^2 + \cdots$. Then
\[
T(x) = \frac{1}{(1-G(x))^2} - \frac{1}{1-G(x)}.
\]
\end{thm}

\begin{proof}
Using $T_0=0$, the recurrence in (\ref{eqTnRecur}), and Theorem \ref{ThmFnGen}, we have
\begin{align*}
T(x)&=\sum_{n=1}^\infty T_n x^n \\
&=\sum_{n=1}^\infty \left(F_n + \sum_{a=1}^k T_{n-a}\right)x^n\\
&=\sum_{n=1}^\infty F_n x^n + \sum_{a=1}^k\sum_{n=1}^\infty T_{n-a}x^n\\
&=F(x)-1 + \sum_{a=1}^k x^a\sum_{n=1}^\infty T_{n-a}x^{n-a}\\
&=\frac{1}{1-G(x)}-1 + \sum_{a=1}^kx^aT(x).
\end{align*}
Solving for $T(x)$ yields $T(x)=1/(1-G(x))^2 - 1/(1-G(x))$, as desired.
\end{proof}


\section{Asymptotics}
\label{SecAsymptotics}

In this section we use the analytic properties of our generating functions to obtain our asymptotic formulas for $F_n$, $T_n$, $A_n$, $C_{n,j}$, and $A_{n,j}$. In this section we use $n$ to stand for a nonnegative integer, $x$ a real number, and $z$ a complex number. In addition, for purposes of nondegeneracy in our arguments, in this section we assume $k>1$. However, it is easily verified by inspection that all of our asymptotic results also hold for $k=1$.
In Section \ref{SecAsymptoticPreliminaries} we review the ideas from analytic combinatorics and establish several facts we need to derive our asymptotic results. In Section \ref{SecAsymptoticDerivations} we prove our asymptotic results.

\subsection{Preliminaries}
\label{SecAsymptoticPreliminaries}

We are going to obtain our asymptotic formulas using the methods of analytic combinatorics \cite{AC,GF}. Namely, we will view the generating functions obtained in Section \ref{SecGenFns} as functions of a complex variable and we will analyze their dominant singularities to give us information about the growth rate of their coefficients. First, we quickly review big-$O$ notation.

\begin{defn}Let Let $\N=\{0,1,2,\ldots\}$ and let $f,g:\N\rightarrow\R$. We write $f(n)=O(g(n))$ if there are positive real numbers $C$ and $N$ for which
\[
|f(n)|\leq C |g(n)|
\]
whenever $n > N$.
\end{defn}
If $f, g,$ and $h$ are functions, whenever we write an equation like
\[
f(n) = h(n) + O(g(n)),
\]
what we mean is
\[
f(n) = h(n) + E(n)
\]
for some function $E$, where $E(n)=O(g(n))$.

The basic observation from complex analysis that we will need is this \cite[Theorem 2.4.3]{GF}.

\begin{thm}Let $f(z)=\sum_{n=0}^\infty a_n z^n$ be analytic in a region containing the origin, and let $z_0\neq 0$ be any singularity of $f(z)$ of smallest modulus. Let $R=|z_0|$ 
and fix a value $\epsilon>0$. Then there exists $N$ such that, for all $n>N$, we have
\[
|a_n| < \left( \frac{1}{R} + \epsilon \right)^n.
\]
Also, for infinitely many $n$ we have
\[
|a_n| > \left(\frac{1}{R} - \epsilon \right)^n.
\]
\end{thm}

The main idea is then the following (see also \cite{AC,GF}). Let $H(z)=h_0+h_1 z + h_2 z^2+\cdots$ be a generating function (in particular, $h_i\geq 0$ for all $i$) with only isolated singularities and whose coefficient growth rate we'd like to understand. We use the standard notation $[z^n]H(z)$ to stand for the coefficient $h_n$ of $z^n$ in $H(z)$. Suppose that, as a function of a complex variable $z$, $H(z)$ is analytic in a region containing the origin and has radius of convergence $R$. Then $H(z)$ has a singularity on the circle $\{z\in \C:|z|=R\}$. Suppose there is only one singularity $z_0$ on this circle (as there will be for our generating functions), and let $S(z)$ be the principal part of $H(z)$ at that singularity. We call $z_0$ the {\em dominant} singularity of $H(z)$. Then the function $H(z)-S(z)$ is analytic in a disk of radius $R'>R$ centered at the origin. Thus, for any fixed $\epsilon>0$, the coefficients of its expansion at the origin cannot grow faster than
\[
\left(\frac{1}{R'} + \epsilon \right)^n
\]
for sufficiently large $n$, and hence
\[
[z^n]H(z) = [z^n]S(z) + O\left( \left(\frac{1}{R'} + \epsilon \right)^n \right).
\]
This is especially good when $R'>1$: In this case we obtain that $[z^n]S(z)$ is a precise asymptotic for $h_n$. (In particular, we obtain that $[z^n]H(z)=[z^n]S(z) + O(A^n)$ for some number $A$ with $0<A<1$.) We remark that even when it is not the case that $R'>1,$ one can still use this formulation to obtain asymptotics---for many examples, see \cite{AC,GF}.

Recall that $G(z)=\sum_{a=1}^k z^a$, $\phi$ denotes the unique positive real solution to
$$
\frac{1}{\phi^{1}} + \frac{1}{\phi^{2}} + \cdots + \frac{1}{\phi^{k}} = 1,
$$
and $\sigma=1/\phi$. Thus $G(\sigma)=1$, so $\sigma$ is a singularity of $F(z)$ and hence of all of our generating functions. In fact $\sigma$ is the dominant singularity of all of our generating functions, and we also have $R'>1$ for all of our generating functions:

\begin{lem}
\label{LemKbonPrecise}
The only number $z\in \C$ for which $G(z)=1$ and $|z|\leq 1$ is $z=\sigma$.
\end{lem}

\begin{proof}
Since we have assumed $k\geq 2$, we have $\phi>1$, so $0<\sigma<1$. Miles \cite{Miles} showed that $M(z)=z^k-z^{k-1}-\cdots-z-1$ has one root of modulus greater than one, and $k-1$ (distinct) roots of modulus strictly smaller than one. It is straightforward to verify that $r\in \C$ is a root of $M(z)$ if and only if $1/r$ is a root of $G(z)-1$, which in turn yields the stated result.
\end{proof}

Thus, the expansion of our generating function $F(z)$ at $z=\sigma$ will play a role in all of our asymptotic derivations. 

\begin{lem}
\label{LemMainExpansion}
The Laurent expansion of
\[
F(z)=\frac{1}{1-G(z)}
\]
at $z=\sigma$ is
\[
F(z)=\frac{-1}{G'(\sigma)(z-\sigma)} + \frac{G''(\sigma)}{2G'(\sigma)^2} + l_1(z-\sigma) + l_2(z-\sigma)^2+\cdots
\]
for some (unimportant) complex coefficients $l_1,l_2,\ldots$.
\end{lem}

\begin{proof}Descartes' rule of signs implies that $z=\sigma$ has multiplicity one as a root of $1-G(z)$. 
Thus the desired expansion is 
\[
\frac{a}{z-\sigma} + b + l_1(z-\sigma) + l_2(z-\sigma)^2+\cdots
\]
for some constants $a,b,l_1,l_2,\ldots$. We need to establish the values of $a$ and $b$. Write $1-G(z)=(z-\sigma)g(z)$ for some polynomial $g(z)$ of which $\sigma$ is not a root. First, $a$ is the residue of $1/(1-G(z))$ at $z=\sigma$, so $a=1/g(\sigma)$. Note that, since $\sigma$ is a positive real number, the coefficients of $g(z)$ are all real numbers. To compute $a$, then, we use the continuity of $g$ along the real line and l'Hospital's rule to obtain 
\begin{equation}
\label{eqgandGprimeAtSigma}
g(\sigma)=\lim_{x\rightarrow \sigma} g(x) = \lim_{x\rightarrow\sigma} \frac{g(x)(x-\sigma)}{x-\sigma} =
\lim_{x\rightarrow\sigma}\frac{1-G(x)}{x-\sigma}=
\lim_{x\rightarrow\sigma}\frac{-G'(x)}{1}=
-G'(\sigma),
\end{equation}
and hence $a=-1/G'(\sigma)$. 

Next, we have that
\[
\frac{1}{g(z)} = \frac{(z-\sigma)}{(z-\sigma)g(z)}= \frac{(z-\sigma)}{1-G(z)} = a+b(z-\sigma) + l_1(z-\sigma)^2+l_2(z-\sigma)^3+\cdots
\]
is analytic in a disk centered at $z=\sigma$, so 
\[
b=\frac{d}{dz}\left( \frac{1}{g(z)} \right)\bigg|_{z=\sigma} = -\frac{g'(\sigma)}{g(\sigma)^2}.
\]
We note that differentiating the equation $g(x)(x-\sigma)=1-G(x)$ and performing a little algebra yields
\[
g'(x)(x-\sigma) = -G'(x)-g(x).
\]
To compute $g'(\sigma)$ we use the continuity of $g'$ along the real line and l'Hospital's rule again:
\begin{align*}
g'(\sigma)=
\lim_{x\rightarrow\sigma}g'(x)=
\lim_{x\rightarrow\sigma}\frac{g'(x)(x-\sigma)}{x-\sigma} = 
\lim_{x\rightarrow\sigma}\frac{-G'(x)-g(x)}{x-\sigma}&=
\lim_{x\rightarrow\sigma}\frac{-G''(x)-g'(x)}{1}\\
&=-G''(\sigma)-g'(\sigma).
\end{align*}
(The use of l'Hospital's rule here is justified by (\ref{eqgandGprimeAtSigma}), which says that $\lim_{x\rightarrow\sigma}-G'(x)-g(x)=-G'(\sigma)-g(\sigma)=0$.) Finally, solving this equation for $g'(\sigma)$ yields
\[
g'(\sigma) = \frac{-G''(\sigma)}{2},
\]
and hence
\[
b=\frac{G''(\sigma)}{2G'(\sigma)^2},
\]
as claimed.
\end{proof}

Finally, we will also need the straightforward facts that, expanding at the origin, we have
\begin{equation}
\label{eqExpansion1}
\frac{1}{z-\sigma} = -\left(\frac{1}{\sigma} + \frac{z}{\sigma^2} + \frac{z^2}{\sigma^3} + \cdots \right),
\end{equation}
and
\begin{equation}
\label{eqExpansion2}
\frac{1}{(z-\sigma)^2} = \frac{1}{\sigma^2} + \frac{2z}{\sigma^3} + \frac{3z^2}{\sigma^4} + \cdots.
\end{equation}

\subsection{Asymptotic formulas}
\label{SecAsymptoticDerivations}

We are now ready to establish our asymptotic formulas. We begin with $F_n$. The asymptotic for $F_n$ stated here is equivalent to the one in \cite{Flores}, where it was derived using the theory of recurrence relations whose characteristic polynomials have no repeated roots. Here we offer an alternate derivation.

\begin{thm}
\label{thmFnAsymptotic}We have 
\[
F_n = \frac{\phi^{n+1}}{G'(\sigma)} + O(A^n)
\]
for some number $A$ with $0<A<1$. 
\end{thm}

\begin{proof}By Lemma \ref{LemMainExpansion}, the expansion of the generating function $F(z)=1/(1-G(z))$ at its dominant singularity, $z=\sigma$, is
\[
-\frac{1}{G'(\sigma)(z-\sigma)} + J(z-\sigma)
\]
where $J(z-\sigma)$ consists of $(z-\sigma)^0$ and higher-order terms. 
The expansion of the principal part of this function,
\[
\frac{-1}{G'(\sigma)(z-\sigma)},
\]
at the origin is, by Equation (\ref{eqExpansion1}),
\[
S(z)=\frac{1}{G'(\sigma)}\left(\frac{1}{\sigma} + \frac{z}{\sigma^2} + \frac{z^2}{\sigma^3}  + \cdots \right).
\]
Thus
\[
[z^n]S(z) = \frac{1}{\sigma^{n+1}G'(\sigma)}
\]
and so
\[
F_n=[z^n]F(z) = \frac{1}{\sigma^{n+1}G'(\sigma)} + O\left(\left(\frac{1}{R'}+\epsilon\right)^n\right)
\]
for any $\epsilon>0$, where $R'$ is the distance from the origin to the singularity of $F(z)$, other than $\sigma$, closest to the origin. Lemma \ref{LemKbonPrecise} implies that $R'>1$, so $\epsilon>0$ can be chosen for which $0<(1/R'+\epsilon)<1$. Since $\phi=1/\sigma$, this establishes our result with $A=1/R'+\epsilon$.
\end{proof}

Next we turn to $T_n$. 

\begin{thm}
\label{thmTnAsymptotic}
We have
\[
T_n = \frac{\phi^{n+2}}{G'(\sigma)^2}(n+1) + \frac{\phi^{n+1}G''(\sigma)}{G'(\sigma)^3} - \frac{\phi^{n+1}}{G'(\sigma)}
+O(A^n)
\]
for some number $A$ with $0<A<1$.
\end{thm}

Since $A_n=T_n/F_n$, if one accepts as given the asymptotic for $A_n$ from Theorem \ref{ThmMyResults}, one can immediately deduce this asymptotic for $T_n$ by multiplying the asymptotic formula for $F_n$ in Theorem \ref{thmFnAsymptotic} by the one for $A_n$ in Theorem \ref{ThmMyResults}. We will also give our own proof of the asymptotic for $A_n$ below which is based on our asymptotic for $T_n$, so we also offer the following direct derivation of Theorem \ref{thmTnAsymptotic}.

\begin{proof}[Proof of Theorem \ref{thmTnAsymptotic}]
By Lemma \ref{LemMainExpansion}, we have that the expansion of the generating function $T(z)={1}/{(1-G(z))^2} - {1}/{(1-G(z))}$ at its dominant singularity, $z=\sigma$, is 
\[
\frac{1}{G'(\sigma)^2 (z-\sigma)^2} + \frac{G'(\sigma)^2-G''(\sigma)}{G'(\sigma)^3(z-\sigma)} + J(z-\sigma),
\]
where $J(z-\sigma)$ consists of $(z-\sigma)^0$ and higher-order terms.
The $z^n$ term of the expansion of the principal part of this function at the origin is, by Equations (\ref{eqExpansion1}) and (\ref{eqExpansion2}),
\[
\frac{(n+1)}{G'(\sigma)^2\sigma^{n+2}} + \frac{G''(\sigma) - G'(\sigma)^2}{G'(\sigma)^3\sigma^{n+1}},
\]
so
\[
T_n = \frac{(n+1)}{G'(\sigma)^2\sigma^{n+2}} + \frac{G''(\sigma) - G'(\sigma)^2}{G'(\sigma)^3\sigma^{n+1}} + O\left(\left(\frac{1}{R'}+\epsilon\right)^n\right)
\]
for any $\epsilon>0$, where $R'$ is the distance from the origin to the singularity of $1/(1-G(z))$, other than $\sigma$, closest to the origin. Again, Lemma \ref{LemKbonPrecise} implies that $R'>1$, so $\epsilon>0$ can be chosen for which $0<(1/R'+\epsilon)<1$, establishing our result with $A=1/R'+\epsilon$.
\end{proof}

Next we turn to $A_n$, whose asymptotic here was proved in a more general setting in \cite[Theorem V.1]{AC}. Our proof is similar to theirs and foreshadows our proof of Theorem \ref{thmAjnAsymptotic}.

\begin{thm}
\label{thmAnAsymptotic}
We have
\[
A_n = \frac{\phi}{G'(\sigma)}(n+1) - 1 + \frac{G''(\sigma)}{G'(\sigma)^2} +O(A^n)
\]
for some number $A$ with $0<A<1$.
\end{thm}

\begin{proof}
From Theorems \ref{thmFnAsymptotic} and \ref{thmTnAsymptotic} we have
\[
F_n = \frac{\phi^{n+1}}{G'(\sigma)} + E(n) =
\frac{\phi^{n+1}}{G'(\sigma)}\left(1+\frac{E(n)G'(\sigma)}{\phi^{n+1}} \right)
\]
and
\[
T_n = \frac{\phi^{n+2}}{G'(\sigma)^2}(n+1) + \frac{\phi^{n+1}G''(\sigma)}{G'(\sigma)^3} - \frac{\phi^{n+1}}{G'(\sigma)}
+E_2(n)
\]
for some functions $E$ and $E_2$ with $E(n)=O(A^n)$ and $E_2(n)=O(A^n)$ for some constant $A$ with $0<A<1$. Let $B$ be any constant with $A<B<1$. Then
\begin{align*}
A_n = \frac{T_n}{F_n} &= 
\left( \frac{\phi}{G'(\sigma)}(n+1) + \frac{G''(\sigma)}{G'(\sigma)^2} - 1 + \frac{E_2(n)G'(\sigma)}{\phi^{n+1}}\right)
\left(\frac{1}{1+E(n)G'(\sigma)/\phi^{n+1}} \right)
\\
&=
\left( \frac{\phi}{G'(\sigma)}(n+1) + \frac{G''(\sigma)}{G'(\sigma)^2} - 1 + E_3(n) \right)
\left(\frac{1}{1+E_4(n)}\right)
\end{align*}
for
\[
E_3(n)= \frac{E_2(n)G'(\sigma)}{\phi^{n+1}}
\]
and
\[
E_4(n)=\frac{E(n)G'(\sigma)}{\phi^{n+1}}.
\]
since $\sigma$ is constant and $\phi>1$, it is immediate that $E_3(n)=O(A^n)$ and $E_4=O(A^n)$. It is therefore also straightforward to obtain 
\[
\frac{1}{1+E_4(n)} = 1+O(A^n).
\]
Thus, we have
\begin{align*}
A_n &= \left( \frac{\phi}{G'(\sigma)}(n+1) + \frac{G''(\sigma)}{G'(\sigma)^2} - 1 + O(A^n) \right)
\left(1+O(A^n)\right).
\end{align*}
Multiplying this out and using $nO(A^n) = O(B^n)$ along with the fact that $G'(\sigma)>0$, $G''(\sigma)>0$, and $\phi>0$ are constants, we obtain
\[
A_n = 
\frac{\phi}{G'(\sigma)}(n+1) - 1 + \frac{G''(\sigma)}{G'(\sigma)^2} + 
O(B^n)
\]
Since $0<B<1$, we are done.
\end{proof}

Next we turn to $C_{n,j}$. 

\begin{thm}
\label{thmCjnAsymptotic}Let $j\in\onetok$. Then we have
\[
C_{n,j} =\frac{\phi^{n+2-j}}{G'(\sigma)^2}(n+1-j) + \frac{\phi^{n+1-j}G''(\sigma)}{G'(\sigma)^3}
+O(A^n)
\]
for some number $A$ with $0<A<1$.
\end{thm}

\begin{proof}We have 
\[
C_{n,j}=[x^n] \frac{x^j}{(1-G(x))^2} = [x^{n-j}] \frac{1}{(1-G(x))^2}.
\]
Lemma \ref{LemMainExpansion} implies that the expansion of $1/{(1-G(z))^2}$ at its dominant singularity, $z=\sigma$, is
\[
\frac{1}{G'(\sigma)^2 (z-\sigma)^2} - \frac{G''(\sigma)}{G'(\sigma)^3(z-\sigma)} + J(z-\sigma),
\]
where $J(z-\sigma)$ consists of $(z-\sigma)^0$ and higher-order terms. The $z^n$ term of the expansion of the principal part of this function at the origin is, by Equations (\ref{eqExpansion1}) and (\ref{eqExpansion2}),
\[
\frac{(n+1)}{G'(\sigma)^2\sigma^{n+2}} + \frac{G''(\sigma)}{G'(\sigma)^3\sigma^{n+1}},
\]
and thus
\[
C_{n,j} = \frac{(n+1-j)}{G'(\sigma)^2\sigma^{n+2-j}} + \frac{G''(\sigma)}{G'(\sigma)^3\sigma^{n+1-j}} + 
O\left(\left(\frac{1}{R'}+\epsilon\right)^n\right)
\]
for any $\epsilon>0$, where $R'$ is the distance from the origin to the singularity of $1/(1-G(z))$, other than $\sigma$, closest to the origin. Again, Lemma \ref{LemKbonPrecise} implies that $R'>1$, so $\epsilon>0$ can be chosen for which $0<(1/R'+\epsilon)<1$, establishing our result with $A=1/R'+\epsilon$.
\end{proof}

Finally, we handle $A_{n,j}$.

\begin{thm}
\label{thmAjnAsymptotic}
Let $j\in \onetok$. Then
\[
A_{n,j} = \frac{\phi^{1-j}}{G'(\sigma)}(n+1-j) + \frac{\phi^{-j}G''(\sigma)}{G'(\sigma)^2} + O(A^n)
\]
for some number $A$ with $0<A<1$.
\end{thm}

\begin{proof}Our proof is similar to the proof of Theorem \ref{thmAnAsymptotic}.
By Theorems \ref{thmFnAsymptotic} and \ref{thmCjnAsymptotic} we have
\[
F_n = \frac{\phi^{n+1}}{G'(\sigma)} + E(n) = \frac{\phi^{n+1}}{G'(\sigma)}\left(1+\frac{E(n)G'(\sigma)}{\phi^{n+1}} \right)
\]
and
\[
C_{n,j} = \frac{\phi^{n+2-j}}{G'(\sigma)^2}(n+1-j) + \frac{\phi^{n+1-j}G''(\sigma)}{G'(\sigma)^3}
+E_2(n)
\]
for some functions $E$ and $E_2$ with $E(n)=O(A^n)$ and $E_2(n)=O(A^n)$ for some constant $A$ with $0<A<1$. Let $B$ be any constant with $A<B<1$. Then
\begin{align*}
A_{n,j}=\frac{C_{n,j}}{F_n} &= \left(\frac{\phi^{1-j}}{G'(\sigma)}(n+1-j) + \frac{\phi^{-j}G''(\sigma)}{G'(\sigma)^2} + \frac{E_2(n)G'(\sigma)}{\phi^{n+1}}\right)
\left( \frac{1}{1+E(n)G'(\sigma)/\phi^{n+1}} \right)\\
&=\left(\frac{\phi^{1-j}}{G'(\sigma)}(n+1-j) + \frac{\phi^{-j}G''(\sigma)}{G'(\sigma)^2} + E_3(n)\right)
\left( \frac{1}{1+E_4(n)} \right)
\end{align*}
for
\[
E_3(n)= \frac{E_2(n)G'(\sigma)}{\phi^{n+1}}
\]
and
\[
E_4(n)=\frac{E(n)G'(\sigma)}{\phi^{n+1}}.
\]
since $\sigma$ is constant and $\phi>1$, it is immediate that $E_3(n)=O(A^n)$ and $E_4=O(A^n)$. It is therefore also straightforward to obtain 
\[
\frac{1}{1+E_4(n)} = 1+O(A^n).
\]
Thus, we have
\[
A_{n,j}=\left(\frac{\phi^{1-j}}{G'(\sigma)}(n+1-j) + \frac{\phi^{-j}G''(\sigma)}{G'(\sigma)^2} + O(A^n)\right)
\left( 1+O(A^n) \right).
\]
Multiplying this out and using $nO(A^n)=O(B^n)$ along with the fact that $G'(\sigma)>0$, $G''(\sigma)>0$, and $\phi>0$ are constants, we obtain
\[
A_{n,j}=\frac{\phi^{1-j}}{G'(\sigma)}(n+1-j) + \frac{\phi^{-j}G''(\sigma)}{G'(\sigma)^2} + O(B^n).
\]
Since $0<B<1$, we are done.

\end{proof}


\section{Tables for small values of $n$ and $k$}
\label{SecExampleValues}
In this section we give tables of values for $F_n$, $T_n$, $A_n$, $C_{n,j}$, and $A_{n,j}$ and our asymptotic approximations to them for small values of $n$ and $k$. 

First, Table \ref{TableCnjExamples} gives the values of $C_{n,j}$ for $j\in \onetok$, for $k=2,3,4$ and  $n=0$ to $n=20$. As we mentioned after the proof of Theorem \ref{ThmCjnGen}, for a fixed value of $k$, the sequences $\{C_{n,j}\}_{n=0}^\infty$ for $j\in \onetok$ are all just shifts of the same sequence.

\begin{small}
\begin{table}[ht]
\caption{Values of $C_{n,j}$ for $k=2,3,4$}\label{TableCnjExamples}
\begin{minipage}[c]{0.22\linewidth}\centering
\begin{tabular}{|l|ll|}
\hline
\multicolumn{3}{|c|}{$k=2$}\\
\hline
$n$ & $C_{n,1}$ & $C_{n,2}$ \\
\hline
0&0&0\\
1&1&0\\
2&2&1\\
3&5&2\\
4&10&5\\
5&20&10\\
6&38&20\\
7&71&38\\
8&130&71\\
9&235&130\\
10&420&235\\
11&744&420\\
12&1308&744\\
13&2285&1308\\
14&3970&2285\\
15&6865&3970\\
16&11822&6865\\
17&20284&11822\\
18&34690&20284\\
19&59155&34690\\
20&100610&59155\\
\hline
\end{tabular}
\end{minipage}
\begin{minipage}[c]{.33\linewidth}\centering
\begin{tabular}{|l|lll|}
\hline
\multicolumn{4}{|c|}{$k=3$}\\
\hline
$n$ & $C_{n,1}$ & $C_{n,2}$ & $C_{n,3}$ \\
\hline
0&0&0&0\\
1&1&0&0\\
2&2&1&0\\
3&5&2&1\\
4&12&5&2\\
5&26&12&5\\
6&56&26&12\\
7&118&56&26\\
8&244&118&56\\
9&499&244&118\\
10&1010&499&244\\
11&2027&1010&499\\
12&4040&2027&1010\\
13&8004&4040&2027\\
14&15776&8004&4040\\
15&30956&15776&8004\\
16&60504&30956&15776\\
17&117845&60504&30956\\
18&228818&117845&60504\\
19&443057&228818&117845\\
20&855732&443057&228818\\
\hline
\end{tabular}
\end{minipage}
\begin{minipage}[c]{.43\linewidth}\centering
\begin{tabular}{|l|llll|}
\hline
\multicolumn{5}{|c|}{$k=4$}\\
\hline
$n$ & $C_{n,1}$ & $C_{n,2}$ & $C_{n,3}$ & $C_{n,4}$ \\
\hline
0&0&0&0&0\\
1&1&0&0&0\\
2&2&1&0&0\\
3&5&2&1&0\\
4&12&5&2&1\\
5&28&12&5&2\\
6&62&28&12&5\\
7&136&62&28&12\\
8&294&136&62&28\\
9&628&294&136&62\\
10&1328&628&294&136\\
11&2787&1328&628&294\\
12&5810&2787&1328&628\\
13&12043&5810&2787&1328\\
14&24840&12043&5810&2787\\
15&51016&24840&12043&5810\\
16&104380&51016&24840&12043\\
17&212848&104380&51016&24840\\
18&432732&212848&104380&51016\\
19&877400&432732&212848&104380\\
20&1774672&877400&432732&212848\\
\hline
\end{tabular}
\end{minipage}
\end{table}
\end{small}

Next, Tables \ref{Table2}, \ref{Table3}, and \ref{Table4} contain the values of $F_n$, $T_n$, $A_n$, $C_{n,1}$, and $A_{n,1}$ and our asymptotic formula approximations to them for $k=2,3,4$ and $n=0$ to $n=15$. Table \ref{Table2} corresponds to $k=2$, Table \ref{Table3} corresponds to $k=3$, and Table \ref{Table4} corresponds to $k=4$. Entries in these tables are rounded to 3 decimal places. In these tables, the approximating formulas we are using are those given in Theorem \ref{ThmMyResults} without the $O(A^n)$ terms. So, for instance, our approximation to $F_n$ is given by 
\[
F_n \approx \frac{\phi^{n+1}}{G'(\sigma)}.
\]

\begin{table}[!ht]
\caption{Values of $F_n$, $T_n$, $A_n$, $C_{n,1}$, and $A_{n,1}$ for $k=2$; $\phi=(1+\sqrt{5})/2$}\label{Table2}
\begin{tabular}{|l|ll|ll|ll|ll|ll|} 
\hline 
$n$ & $F_n$ & Appr. & $T_n$ & Appr. & $A_n$ & Appr. & $C_{n,1}$ & Appr. & $A_{n,1}$ & Appr. \\
\hline
$0$&$1$&$0.724$&$0$&$0.089$&$0.0$&$0.124$&$0$&$0.179$&$0.0$&$0.247$\\
$1$&$1$&$1.171$&$1$&$0.992$&$1.0$&$0.847$&$1$&$0.813$&$1.0$&$0.694$\\
$2$&$2$&$1.894$&$3$&$2.976$&$1.5$&$1.571$&$2$&$2.163$&$1.0$&$1.142$\\
$3$&$3$&$3.065$&$7$&$7.033$&$2.333$&$2.294$&$5$&$4.87$&$1.667$&$1.589$\\
$4$&$5$&$4.96$&$15$&$14.968$&$3.0$&$3.018$&$10$&$10.098$&$2.0$&$2.036$\\
$5$&$8$&$8.025$&$30$&$30.026$&$3.75$&$3.742$&$20$&$19.928$&$2.5$&$2.483$\\
$6$&$13$&$12.985$&$58$&$57.979$&$4.462$&$4.465$&$38$&$38.051$&$2.923$&$2.93$\\
$7$&$21$&$21.01$&$109$&$109.015$&$5.19$&$5.189$&$71$&$70.964$&$3.381$&$3.378$\\
$8$&$34$&$33.994$&$201$&$200.989$&$5.912$&$5.912$&$130$&$130.025$&$3.824$&$3.825$\\
$9$&$55$&$55.004$&$365$&$365.008$&$6.636$&$6.636$&$235$&$234.983$&$4.273$&$4.272$\\
$10$&$89$&$88.998$&$655$&$654.995$&$7.36$&$7.36$&$420$&$420.012$&$4.719$&$4.719$\\
$11$&$144$&$144.001$&$1164$&$1164.004$&$8.083$&$8.083$&$744$&$743.992$&$5.167$&$5.167$\\
$12$&$233$&$232.999$&$2052$&$2051.997$&$8.807$&$8.807$&$1308$&$1308.005$&$5.614$&$5.614$\\
$13$&$377$&$377.001$&$3593$&$3593.002$&$9.531$&$9.53$&$2285$&$2284.997$&$6.061$&$6.061$\\
$14$&$610$&$610.0$&$6255$&$6254.999$&$10.254$&$10.254$&$3970$&$3970.002$&$6.508$&$6.508$\\
$15$&$987$&$987.0$&$10835$&$10835.001$&$10.978$&$10.978$&$6865$&$6864.999$&$6.955$&$6.955$\\
\hline
\end{tabular}
\end{table}

\begin{table}[ht]
\caption{Values of $F_n$, $T_n$, $A_n$, $C_{n,1}$, and $A_{n,1}$ for $k=3$; $\phi\approx1.8392868$}\label{Table3}
\begin{tabular}{|l|ll|ll|ll|ll|ll|} 
\hline 
$n$ & $F_n$ & Appr. & $T_n$ & Appr. & $A_n$ & Appr. & $C_{n,1}$ & Appr. & $A_{n,1}$ & Appr. \\
\hline
$0$&$1$&$0.618$&$0$&$0.132$&$0.0$&$0.213$&$0$&$0.2$&$0.0$&$0.323$\\
$1$&$1$&$1.137$&$1$&$0.946$&$1.0$&$0.832$&$1$&$0.75$&$1.0$&$0.66$\\
$2$&$2$&$2.092$&$3$&$3.034$&$1.5$&$1.45$&$2$&$2.083$&$1.0$&$0.996$\\
$3$&$4$&$3.848$&$8$&$7.96$&$2.0$&$2.069$&$5$&$5.126$&$1.25$&$1.332$\\
$4$&$7$&$7.078$&$19$&$19.017$&$2.714$&$2.687$&$12$&$11.808$&$1.714$&$1.668$\\
$5$&$13$&$13.018$&$43$&$43.028$&$3.308$&$3.305$&$26$&$26.095$&$2.0$&$2.005$\\
$6$&$24$&$23.943$&$94$&$93.948$&$3.917$&$3.924$&$56$&$56.046$&$2.333$&$2.341$\\
$7$&$44$&$44.038$&$200$&$200.032$&$4.545$&$4.542$&$118$&$117.891$&$2.682$&$2.677$\\
$8$&$81$&$80.999$&$418$&$418.008$&$5.16$&$5.161$&$244$&$244.07$&$3.012$&$3.013$\\
$9$&$149$&$148.98$&$861$&$860.968$&$5.779$&$5.779$&$499$&$499.007$&$3.349$&$3.349$\\
$10$&$274$&$274.017$&$1753$&$1753.025$&$6.398$&$6.397$&$1010$&$1009.948$&$3.686$&$3.686$\\
$11$&$504$&$503.996$&$3536$&$3535.998$&$7.016$&$7.016$&$2027$&$2027.043$&$4.022$&$4.022$\\
$12$&$927$&$926.994$&$7077$&$7076.985$&$7.634$&$7.634$&$4040$&$4039.994$&$4.358$&$4.358$\\
$13$&$1705$&$1705.007$&$14071$&$14071.015$&$8.253$&$8.253$&$8004$&$8003.979$&$4.694$&$4.694$\\
$14$&$3136$&$3135.997$&$27820$&$27819.995$&$8.871$&$8.871$&$15776$&$15776.023$&$5.031$&$5.031$\\
$15$&$5768$&$5767.998$&$54736$&$54735.994$&$9.49$&$9.49$&$30956$&$30955.993$&$5.367$&$5.367$\\
\hline
\end{tabular}  
\end{table}

\begin{table}[ht]
\caption{Values of $F_n$, $T_n$, $A_n$, $C_{n,1}$, and $A_{n,1}$ for $k=4$; $\phi\approx1.92756198$}\label{Table4}
\begin{tabular}{|l|ll|ll|ll|ll|ll|} 
\hline 
$n$ & $F_n$ & Appr. & $T_n$ & Appr. & $A_n$ & Appr. & $C_{n,1}$ & Appr. & $A_{n,1}$ & Appr. \\
\hline
$0$&$1$&$0.566$&$0$&$0.162$&$0.0$&$0.287$&$0$&$0.212$&$0.0$&$0.374$\\
$1$&$1$&$1.092$&$1$&$0.931$&$1.0$&$0.853$&$1$&$0.729$&$1.0$&$0.667$\\
$2$&$2$&$2.104$&$3$&$2.986$&$1.5$&$1.419$&$2$&$2.023$&$1.0$&$0.961$\\
$3$&$4$&$4.056$&$8$&$8.053$&$2.0$&$1.986$&$5$&$5.091$&$1.25$&$1.255$\\
$4$&$8$&$7.818$&$20$&$19.951$&$2.5$&$2.552$&$12$&$12.11$&$1.5$&$1.549$\\
$5$&$15$&$15.07$&$47$&$46.993$&$3.133$&$3.118$&$28$&$27.77$&$1.867$&$1.843$\\
$6$&$29$&$29.049$&$107$&$107.033$&$3.69$&$3.685$&$62$&$62.063$&$2.138$&$2.136$\\
$7$&$56$&$55.994$&$238$&$238.024$&$4.25$&$4.251$&$136$&$136.082$&$2.429$&$2.43$\\
$8$&$108$&$107.931$&$520$&$519.932$&$4.815$&$4.817$&$294$&$294.017$&$2.722$&$2.724$\\
$9$&$208$&$208.044$&$1120$&$1120.024$&$5.385$&$5.384$&$628$&$627.863$&$3.019$&$3.018$\\
$10$&$401$&$401.017$&$2386$&$2386.03$&$5.95$&$5.95$&$1328$&$1328.068$&$3.312$&$3.312$\\
$11$&$773$&$772.985$&$5037$&$5036.995$&$6.516$&$6.516$&$2787$&$2787.047$&$3.605$&$3.606$\\
$12$&$1490$&$1489.977$&$10553$&$10552.957$&$7.083$&$7.083$&$5810$&$5809.98$&$3.899$&$3.899$\\
$13$&$2872$&$2872.024$&$21968$&$21968.029$&$7.649$&$7.649$&$12043$&$12042.935$&$4.193$&$4.193$\\
$14$&$5536$&$5536.004$&$45480$&$45480.014$&$8.215$&$8.215$&$24840$&$24840.053$&$4.487$&$4.487$\\
$15$&$10671$&$10670.99$&$93709$&$93708.986$&$8.782$&$8.782$&$51016$&$51016.018$&$4.781$&$4.781$\\
\hline
\end{tabular} 
\end{table}

We also remark that for $k=2$, Binet's formula leads to the same approximating asymptotic for $F_n$---for $n\geq 0$, Binet's formula for the $n$th Fibonacci number yields
\[
F_n = \frac{1}{\sqrt{5}}\left(\phi^{n+1} - \left(\frac{-1}{\phi}\right)^{n+1}\right).
\]
Since $0<1/\phi<1$, this gives the asymptotic formula $F_n\sim \phi^{n+1}/\sqrt{5}$---specifically, we have
\[
F_n = \frac{\phi^{n+1}}{\sqrt{5}} + E(n),
\]
where the error term
\[
E(n) = -\frac{1}{\sqrt{5}}\left(\frac{-1}{\phi}\right)^{n+1}
\]
approaches $0$ at an exponential rate. Indeed, this is the same as our asymptotic $F_n\sim {\phi^{n+1}}/{G'(\sigma)}$ for $k=2$, because for $k=2$ we have $\phi=(1+\sqrt{5})/2$ and $G'(x)=1+2x$, so $G'(\sigma)=G'(1/\phi)=\sqrt{5}$.

Finally, because all of our asymptotic formulas have absolute error terms that decay exponentially, in Tables \ref{Table2}, \ref{Table3}, and \ref{Table4} we obtain excellent approximations to our values $F_n$, $T_n$, $A_n$, $C_{n,j}$, and $A_{n,j}$ even for small values of $n$.

\bibliographystyle{plain}
\bibliography{CompositionsBibAbbrev}


\end{document}